\documentclass[12pt,a4paper,reqno]{amsart}

\usepackage{amssymb}
\usepackage{euscript}
\usepackage{eulervm}

\newtheorem{proposition}{Proposition}
\newtheorem*{MO}{Mazur-Orlicz Theorem}

\newtheorem*{corollary}{Corollary}
\newtheorem{lemma}{Lemma}

\newcommand{\plus}[1]{{#1}^+}  
\newcommand{\diff}{\triangle\!} 
\newcommand{\stirling}[2]{\left\{ {#1 \atop #2} \right\}}  
\newcommand{\stirlingfirst}[2]{\left[ {#1 \atop #2} \right]}  
\newcommand{\falling}[2]{#1^{\underline{#2}}}

\newcommand{\R}{\ensuremath{\mathbb{R}}}

\newcommand{\N}{\ensuremath{\mathbb{N}}}
\newcommand{\Q}{\ensuremath{\mathbb{Q}}}

\begin{document}
\title{ Positive Polynomials on Riesz Spaces}

\author{James Cruickshank}
\email{james.cruickshank@nuigalway.ie}
\address{National University of Ireland Galway}
\author{John Loane}
\email{john.loane@dkit.ie}
\address{Dundalk  Institute of Technology}
\author{Raymond A. Ryan}
\email{ray.ryan@nuigalway.ie}
\address{National University of Ireland Galway}
\date{\today}

\begin{abstract}
We prove some properties of positive  polynomial mappings between
Riesz spaces, using finite difference calculus. 
We establish the polynomial analogue of the classical
result that positive, additive mappings are linear. And we prove a
polynomial
version of the Kantorovich extension theorem.
\end{abstract}

\maketitle

\section{Introduction}
The study of polynomial mappings on Riesz spaces is relevant
in a number of areas, including orthogonal
additivity and concavifications \cite{BLL,S}, symmetric Fremlin tensor 
products \cite{Bu}, monomial expansions for analytic
functions \cite{GR} and Hahn-Banach extension theorems
for polynomial functions \cite{Lo1,Lo2}.
The purpose of this paper is to establish some basic results about
positive polynomials.  
We extend some classical results from the linear
to the polynomial setting, including the Kantorovich extension
theorem.
A polynomial mapping is made up of a sum of homogeneous components,
each of which is generated by a multilinear mapping.  So our approach
is to go from linear to multilinear mappings, then to homogeneous polynomials
and finally to polynomials.  This process is aided by the use of
some techniques from finite difference calculus.

Let $E$, $F$ be vector spaces over the real numbers.
A mapping $P_k\colon E \to F$ is called a \emph{$k$-homogeneous polynomial}
if there exists a $k$-linear mapping $A_k\colon E^k \to F$ such that
$P_k(x) = A_k(x,\dots,x)$ for every $x\in E$.  We write this as
$P_k(x) = A_k(x^k)$  to indicate the $k$-fold repetition of the variable $x$.

If we require only that $A_k$ be additive rather than linear
in each variable,  we get the definition of a \emph{$k$-homogeneous $\Q$-polynomial}.
This
is equivalent to $A_k$ being $k$-linear with respect to
the field of rational numbers. Even in the case $k=1$ it is a standard result that
when $E=F=\R$, there exist additive (rational linear) mappings that are not linear \cite[pages~128--130]{K}.

A mapping $P\colon E\to F$ is called a \emph{polynomial of degree $m$}
(respectively, a \emph{$\Q$-polynomial of degree $m$})
if there exist  $k$-homogeneous polynomials $P_k$
(respectively, $k$-homogeneous $\Q$-polynomials $P_k$), for
$0\le k\le m$, with $P_m\neq 0$, such that $P=P_0 + \dots + P_m$.

If $P_k\colon E\to F$ is a  $k$-homogeneous $\Q$-polynomial that is generated by the
$k$-additive mapping $A_k$, we may assume without loss
of generality that $A_k$ is symmetric.  In this case, $A_k$ is uniquely determined by $P_k$
and we write $P_k = \hat{A}_k$.  Furthermore, $A_k$ can be recovered  from $P_k$
by means of the Polarization Formula:
\begin{equation} \label{polarization}
A_k(x_1,\dots,x_k) = \frac{1}{2^k k!} \sum_{\varepsilon_j=\pm1} \varepsilon_1 \dots \varepsilon_k
P_k(\varepsilon_1x_1 +\dots+ \varepsilon_k x_k)\,.
\end{equation}
This is easily proved by expanding the right hand side.  This is not the only such formula.  We also
have the Mazur-Orlicz Polarization Formula \cite{MO}:
\begin{equation}\label{MO polarization}
A_k(x_1,\dots,x_k) = \frac{1}{k!}\sum_{\delta_i = 0,1} (-1)^{k-\sum \delta_i} \,P_k(x+\delta_1 x_1 + \dots + \delta_k x_k) \,.
\end{equation}
Here $x\in E$ is arbitrary. The proof is given later.
The first formula is a special case of this one.


\section{Finite Difference Calculus for Polynomials}
\bigskip
Finite difference calculus is an effective tool for analysing polynomial 
mappings on vector spaces
and provides some useful insights.  We start with a brief review.

Let $f\colon E\to F$ be any mapping.
For $x,h\in E$ the forward difference $\diff f(x;h)$ is $f(x+h)-f(x)$.  Higher order differences
are defined recursively by
$$
\diff ^{n+1}f(x;h_1,\dots,h_{n+1})= \diff
\bigl(\diff ^n f(x;h_1,\dots,h_n)\bigr)(x;h_{n+1}) \,.
$$
It is easy to see that
\begin{align} \label{mixed}
\diff ^{n}f(x;h_1,\dots,h_{n})=& \sum_{\delta_i = 0,1}
(-1)^{n-\sum \delta_i} \,f(x+\delta_1 h_1 + \dots + \delta_n h_n)\\
=& \sum_{k=0}^n (-1)^{n-k} \negthickspace \negthickspace \negthickspace
\sum_{\substack{\delta_i= 0,1\\  \delta_1+\dots +\delta_n=k}}
f(x+\delta_1h_1 + \dots + \delta_n h_n) \,. \notag
\end{align}
When the increments $h_1=\dots h_n=h$ are equal, we write $\diff^n f(x;h^n)$ instead of $\diff^n f(x;h,\dots,h)$.
We refer to these as \emph{pure differences} and to differences of the form
$\diff^n(f; h_1,\dots,h_n)$ as \emph{mixed differences}..  For pure differences, the preceding formula reduces to
\begin{equation}\label{pure}
\diff ^n f(x;h^n) = \sum_{k=0}^n (-1)^{n-k}\; \binom{n}{k}\,f(x+kh)  \,.
\end{equation}
Inverting this, we get the \emph{Newton expansion:}
\begin{equation}\label{Newton}
f(x+nh) = \sum_{k=0}^n \binom{n}{k} \diff ^k f(x;h^k) \,,
\end{equation}
where we define $\diff^0 f(x;h^0)$ to be $f(x)$.

There is a somewhat surprising relationship between mixed and pure differences.  The following identity shows that every mixed difference
can be expressed as a linear combination of pure differences \cite[p.~418]{K}.
\begin{multline}\label{id}
\diff^n f(x;h_1,\dots,h_n) = \\
\sum_{\delta_j= 0,1}(-1)^{\sum \delta_j}\diff\!^n f\Bigl( \Bigl( x + \sum_{j=1}^n \delta_j h_j\Bigr) ; \Bigl( - \sum_{j=1}^n \frac{\delta_j h_j}{j}\Bigr)^n\Bigr) \,.
\end{multline}

Suppose that $P_k$ is a $k$-homogeneous $\Q$-polynomial, generated by the
symmetric $k$-additive symmetric map $A_k$.
We have
\begin{equation*}
\diff P_k(x;h_1) = A_k((x+h_1)^k) - A_k(x^k)
= \sum_{j_1=1}^{k} \binom{k}{j_1} A_k(x^{k-j_1},h_1^{j_1})  \,.
\end{equation*}
Iterating this, or using (\ref{mixed}), we get the general formula:
\begin{multline}\label{mixdiff}
\diff^n P_k(x;h_1,\dots,h_n) =\\
\sum_{\substack{j_0\ge 0, j_1,\dots,j_n\in\N\\j_0+j_1+\dots+j_n=k }}
\binom{k}{j_0, j_1, \dots ,j_n} A_k (x^{j_0},h_1^{j_1},\dots, h_n^{j_n} )
\end{multline}
for $0\le n\le k$.
If we put $n=k$ we get
\begin{equation*}\label{MOProof}
\diff^k P_k(x;h_1,\dots,h_k) = k! \, A_k(h_1,\dots,h_k) \,,
\end{equation*}
which, combined with (\ref{mixed}), gives the Mazur-Orlicz Polarization Formula~(\ref{MO polarization}).
Since the $k$-th differences of a $k$-homogeneous polynomial at the point $x$ do not depend on $x$,
we see that $\diff^n P_k (x;h_1,\dots,h_n)=0$ for $n>k$.

Next, we  find the pure differences of the $k$-homogeneous polynomial $P_k$.
Using~(\ref{pure}), we have
\begin{align*}
\diff^n P_k(x;h^n) &= \sum_{i=0}^n (-1)^{n-i}\; \binom{n}{i}\,P_k(x+ih)\\
&= \sum_{i=0}^n (-1)^{n-i} \binom{n}{i} \sum_{j=0}^k \binom{k}{j}
i^j A_k (x^{k-j},h^j)
\\
&= \sum_{j=0}^k \binom{k}{j} \biggl(\sum_{i=0}^n (-1)^{n-i} \binom{n}{i}
i^j \biggr) A_k (x^{k-j},h^j ) \\
&= n!\sum_{j=0}^k \binom{k}{j} \stirling{j}{n} A_k (x^{k-j},h^j)
\end{align*}
where
\begin{equation*}
\stirling{j}{n} = \frac{1}{n!}\sum_{i=0}^n (-1)^{n-i} \binom{n}{i} i^j
\end{equation*}
are the Stirling numbers of the second kind~\cite{GKP}.
(We follow the usual convention that $i^j=1$ when 
$i=j=0$.)
Summarizing, we have
\begin{equation}
\diff^n P_k(x;h^n) = n!\sum_{j=n}^k
\binom{k}{j} \stirling{j}{n} A_k (x^{k-j},h^j) \,.
\end{equation}
Note that this summation ranges over $n\le j\le k$, since the Stirling
numbers $\stirling{j}{n}$ are zero when $j<n$.

In particular, taking $x=0$ we have
\begin{equation}\label{riordan}
\diff^n P_k(0;h^n) = n! \stirling{k}{n} P_k(h)\,.
\end{equation}
This is a general form of the formula for differences of a scalar monomial given in~\cite[p.~202]{Ri}.

Now let $P = P_0+\dots+ P_m$ be a $\Q$-polynomial of degree $m$.
The $k$-homogeneous components $P_k$ are of course uniquely determined by $P$.
It is sometimes necessary to have a means by which  one of the homogeneous components  can be expressed
explicitly in terms of $P$.  There are several ways to do this.

We begin with an interpolation procedure \cite{MO}.  
If $q(t)= c_0+\dots+c_mt^m$ is a polynomial in 
$t\in\R$  of degree $m$, then  from
\begin{equation*}
q(j) = \sum_{k=0}^m j^k c_k, \quad 0\le j\le m \,.
\end{equation*}
we get
\begin{equation*}
c_k = \sum_{j=0}^m \alpha_{kj} q(j)\,,
\end{equation*}
where $(\alpha_{kj})$ is the inverse of the $(m+1)\times(m+1)$
Vandermonde matrix $(j^k)$.
It follows easily that if $P\colon E\to F$ is a rational polynomial of degree $m$ with homogeneous components
$P_0,\dots, P_m$, then
\begin{equation} \label{interpolation}
P_k(x) = \sum_{j=0}^m \alpha_{kj} P(jx)
\end{equation}
for every $x\in E$ and $k=0,\dots, m$.

It is also possible to use finite differences to extract the homogeneous
components.
We start with an observation by Mazur and Orlicz~\cite{MO}.
Fixing $x\in E$, consider the Newton
expansion~(\ref{Newton}) of $P$ at $0$:
\begin{equation*}
P(nx) = \sum_{j=0}^n \binom{n}{j}\diff^j P(0;x^j)  \,,
\end{equation*}
where $n=1,2,\dots $. Now the differences of $P$ of order $j$ vanish for $j>m$.
And, if $n<m$, then the binomial coefficients vanish when $n<j\le m$.  So we can
replace the variable upper  bound in this sum by $m$:
\begin{equation*}
P(nx) = \sum_{j=0}^m \binom{n}{j}\diff^j P(0;x^j)
\end{equation*}
for every $n\in \N$. We  have
\begin{equation} \label{stirling1}
\binom{n}{j} = \frac{ \falling{n}{j} }{j!}
= \frac{1}{j!} \sum_{k=0}^j\stirlingfirst{j}{k}(-1)^{j-k}n^k \,,
\end{equation}
where   $\falling{n}{j}= n(n-1)\dots (n-j+1)$ denote the  falling factorial powers and
$\stirlingfirst{j}{k}$ are  the Stirling numbers of the first kind~\cite{GKP}.   We now have
\begin{equation*}
P(nx) = \sum_{j=0}^m  \sum_{k=0}^j
\frac{1}{j!} \stirlingfirst{j}{k}(-1)^{j-k}n^k \, \diff^j P(0;x^j) \,.
\end{equation*}
Interchanging the order of summation gives
\begin{equation*}
P(nx) = \sum_{k=0}^m  n^k \, \sum_{j=k}^m
\frac{1}{j!} \stirlingfirst{j}{k}(-1)^{j-k} \diff^j P(0;x^j)
\end{equation*}
for every $n\in \N$.   But we also have
\begin{equation*}
P(nx)= \sum_{k=0}^m n^k\, P_k(x)\,
\end{equation*}
for every $n$.Comparing these, we arrive at a formula for
$P_k(x)$ in terms of the differences of $P$ at the origin:
\begin{equation}\label{extract}
P_k(x) = \sum_{j=k}^m
\frac{1}{j!} \stirlingfirst{j}{k}(-1)^{j-k} \diff^j P(0;x^j) \,.
\end{equation}

This gives one way to express $P_k$ in terms of $P$.  Taking another approach, we apply (\ref{riordan}) to compute the $k$th pure difference
\begin{equation*}
\diff^k P(0, (tx)^k) =
\sum_{j=k}^m k! \stirling{j}{k} P_j(x) t^j
\end{equation*}
for $x\in E$ and $t\in\R$, giving
\begin{equation}\label{limit1}
P_k(x) = \lim_{t\to 0} \frac{\diff^kP(0;(tx)^k)}{k!\,  t^k}\,.
\end{equation}
In the same way, from (\ref{mixdiff}) we get
\begin{equation}\label{limit2}
A_k(x_1,\dots,x_k) =
\lim_{t\to 0} \frac{\diff^k P(0;tx_1,\dots,tx_k)}{k!\, t^k}\,.
\end{equation}

When $P$ is vector valued, these limits are taken with respect to
the finest locally convex topology.

If $P$ is a $\Q$-polynomial of degree $m$, then all differences of $P$
of order $m+1$ vanish.  Mazur and Orlicz~\cite{MO} showed that
this condition is also sufficient.

\begin{MO}\cite{MO}
	Let $E$, $F$ be vector spaces.
	A mapping $P\colon E\to F$  is a $\Q$-polynomial of degree $m$
	or less if and only if  it satisfies the condition $\diff^{m+1} P(x;h^{m+1}) \allowbreak =0$ for all $x,h\in E$.
\end{MO}

This result has been rediscovered  by a number of authors in various settings.  See for example
\cite{McK,Hy,Sze,Van}. 
Note that, by virtue of the identity~(\ref{id}), the condition that all
the pure differences $\diff^{m+1} P(x;h^{m+1})$ vanish is equivalent to the
vanishing of the mixed differences $\diff^{k+1} P(x;h_1,\dots,h_{k+1})$  for all $x,h_1,\dots,h_{k+1}
\in E$.


\section{Positive Polynomials on Riesz Spaces}

Now let $E$ and $F$ be Riesz spaces.  All the Riesz spaces considered are assumed to be archimedean.
A $k$-homogeneous $\Q$-polynomial $P_k=\hat{A}_k$
is said to be \emph{positive} if the symmetric $k$-additive mapping $A_k$ is positive in each variable.  In other words,
we  have $A_k(x_1,\dots,x_k)\ge 0$ for all $x_1,\dots,x_k\ge 0$.  It follows from ($\ref{mixdiff}$) that $P_k$
is positive if and only if
\begin{equation*}
\diff^n P_k(x;h_1,\dots,h_n) \ge 0
\end{equation*}
for all $n=0,1,\dots$ and all $x,h_1,\dots,h_n \ge 0$.  Taking $n=1$, it follows that $P_k$ is 
positive and monotone on the positive cone
of $E$: if $0\le x\le y$, then $0\le P_k(x)\le P_k(y)$.  However, 
positivity and monotonicity on the positive cone 
are not sufficient to guarantee positivity~\cite{Lo2}.
We note too that in the above condition on the differences of $P_k$,
we cannot replace the mixed differences by pure differences. 
A homogeneous polynomial on $\R^n$ is positive if and only if
all the coefficients in its monomial expansion are nonnegative.
Thus the $3$-homogeneous polynomial on $\R^3$ given by
\begin{equation*}
P(x) = x_1^3 + x_2^3 + x_3^3 +3x_1^2(x_2 +x_3) +3x_2^2(x_1 +x_3)
+3x_3^2(x_1 + x_2) -6x_1x_2x_3
\end{equation*}
is not positve.  But straightforward calculations show that 
$\diff^n P(x;h^n) \ge 0$ for all $x,h\ge 0$ and all $n\ge 0$.

A $\Q$-polynomial $P=P_0+ \dots +P_m$ of degree $m$ is said to be positive if each of its homogeneous
components $P_k$ is positive.

\begin{proposition}
	Let $E$, $F$ be Riesz spaces and let  $P\colon E\to F$ be a $\Q$-polynomial.
	The following are equivalent:
	\begin{itemize}
		\item[(a)] $P$ is positive.
		\item[(b)] $\diff^n P (x;h_1,\dots,h_n) \ge 0$ for all $n$ and  for all $x,h_1,\dots,h_n\ge 0$ in $E$.
		\item[(c)] $\diff^n P (0;h_1,\dots,h_n) \ge 0$ for all $n$ and for all $h_1,\dots,h_n\ge 0$ in $E$.
	\end{itemize}
\end{proposition}
\begin{proof}
	(a) $\implies$ (b) follows immediately from the remarks preceding the statement
	of the proposition and (b) $\implies$ (c) is trivial.
	
	Suppose that $P$ satisfies (c).  For the $k$-th
	component $P_k= \hat{A}_k$ and for $x_1,\dots,x_k\ge 0$, 
	we have, using~(\ref{limit2}),
	\begin{equation*}
	A_k(x_1,\dots,x_k) =
	\lim_{t\to 0+} \frac{\diff^k P(0;tx_1,\dots,tx_k)}{k!\, t^k}\,.
	\end{equation*}
	and so each component of $P$ is positive.
	
\end{proof}

Bochnak and Siciak \cite[Corollary 3]{BS} showed that a mapping
$P\colon E\to F$ between vector spaces is a polynomial if and only
if $P$ is a polynomial on every affine line in $E$.  This result
does not extend to positive polynomials on Riesz spaces.  The example
given before the last proposition shows that it is possible for a 
non-positive polynomial to be positive on every affine line.

A $1$-homogeneous $\Q$-polynomial is just an additive mapping, which need not be  linear.
However, every positive additive mapping (with an archimedean range) is linear.  The same is true for polynomials.  Before stating the next result, we recall that all the Riesz spaces in question are assumed to be archimedean.

\begin{proposition}\label{PosProp}
	Every positive $\Q$-polynomial between Riesz spaces is a polynomial.
\end{proposition}

\begin{proof}
	Let $P=P_0 + \dots + P_m\colon E\to F$ be a positive $\Q$-polynomial.  
	Each homogeneous component
	$P_k= \hat{A}_k$ is positive.  For each $j$ between $1$ and $k$,
	if we fix positive $x_1,\dots,x_{j-1}, x_{j+1}, \dots x_k\in E$, then $x_j\mapsto A_k(x_1,\dots,x_k)$
	is a positive additive mapping from $E$ into $F$ and so is linear.
	Thus $A_k$ is linear in each variable and so
	$P_k$ is a $k$-homogeneous polynomial.  Therefore $P$ is a polynomial.
\end{proof}

More generally, every order bounded additive
mapping from a  Riesz space into a Dedekind complete Riesz space is linear
(see, for example, \cite{EE}).  We recall that a mapping from $E$ to $ F$
is said to be order bounded if
it maps every order interval in $E$ into an order bounded subset of $F$.

\begin{proposition}
	Let $E$, $F$ be Riesz spaces and suppose that $F$ is Dedekind-complete.
	Then every order bounded $\Q$-polynomial from $E$ into $F$ is a polynomial.
\end{proposition}

\begin{proof}
	Let $P=P_0 + \dots + P_m\colon E\to F$ be an order bounded $\Q$-polynomial.
	For each $k$, it follows from~(\ref{interpolation}) that $P_k(x)$ is a 
	linear combination
	of the values $P(jx)$, $0\le j \le m$, with coefficients that are independent of $x$.
	Hence $P_k$ is also order bounded.  If $P_k$ is generated
	by the symmetric $k$-additive mapping $A_k$, then by the polarization 
	formula, $A_k$
	is an order bounded function of each of its $k$ variables.  Therefore $A_k$ is linear
	in each variable and so $P_k$ is a $k$-homogeneous polynomial.  Thus $P$ is a polynomial.
\end{proof}

The Kantorovich extension theorem \cite[Theorem 1.10]{AB} is a fundamental
tool in the linear theory.  It states that an additive mapping
$T\colon \plus{E} \to \plus{F}$ between the positive cones of two
Riesz spaces extends to a unique positive linear mapping from $E$
into $F$.  In order to prove an analogous result for polynomial
mappings, we have to find a way to express the appropriate forms
of the additivity and positivity properties.  This can be done using finite 
differences.

We start with a result for $k$-homogeneous mappings.

\begin{lemma}
	Let $E$, $F$, be Riesz spaces and let $f\colon \plus{E}\to
	F$ be a mapping that satisfies
	\begin{enumerate}
		\item[(i)]  
		$\diff ^{k+1} f(0;h^{k+1})=0$ for all $h\in \plus{E}$, 
		\item[(ii)] 
		$f(nx)= n^k f(x)$ for all $x\in\plus{E}$, $n\in\N$\,,
	\end{enumerate}
	for some $k\ge 0$.
	Then $f$ extends to a unique  $k$-homogeneous 
	$\Q$-polynomial
	$P\colon E \to F$.
\end{lemma}

\begin{proof}
	We use a similar line of argument to that in \cite[Satz I.]{MO}.  The case $k=0$ is trivial, 
	so assume that $k\ge 1$.
	We claim that condition (i) implies that	
	\begin{equation*}
	\diff^{k+1} f(0,h_1,\dots,h_{k+1}) =0
	\end{equation*}
	for every $h_1,\dots,h_{k+1} \ge 0$.
	Note that this does not follow from (\ref{id}), since
	we are now dealing with positive increments only.
	We take the Newton expansion of $f(x+nh)$ where 
	$x,h\ge 0$ and $n\in\N$.  As 
	the differences of $f$ of order $j$ vanish for $j>k$
	and the binomial coefficients vanish when $n<j\le k$,
	we have
	\begin{equation*}
	f(x+nh)= \sum_{j=0}^k \binom{n}{j}\diff^j f(x;h^j)\,,
	\end{equation*} 
	which is a polynomial in $n$ of degree $k$ or less. 
	Applying this argument a second time, we see that
	$f(x+n_1h_1+n_2h_2)$ is a polynomial in
	the variables $n_1$, $n_2$.
	Iterating this $k$ times with $x=0$, we find that 
	$f(n_1h_1+ \dots +n_{k+1}h_{k+1})$ is a 
	polynomial in $n_1, \dots,n_{k+1}\in \N$.
	Fixing $h_1,\dots,h_{k+1} \ge 0$, let
	$g\colon \N^{k+1} \to F$ be the polynomial
	given by $g(n_1,\dots,n_{k+1}) =
	f(n_1h_1+ \dots + n_{k+1}h_{k+1})$.
	It follows by condition (ii) that $g$ is $k$-homogeneous.
	Now 
	\begin{equation*}
	\diff ^{{k+1}}f(0;h_1,\dots,h_{k+1})= \sum_{\delta_i = 0,1}
	(-1)^{k+1-\sum \delta_i} \,f(\delta_1 h_1 + \dots + \delta_{k+1} h_{k+1})
	\end{equation*}
	is the same as the ($k+1$)st difference of $g$ at $0$ with 
	increments $e_1,\dots ,e_{k+1}$. Since $g$ is a $k$-homogeneous
	polynomial, these differences vanish.  This establishes our claim.
	
	It follows that the mapping
	\begin{equation*}
	A(x_1,\dots,x_k) =
	\frac{1}{k!} \diff^kf(0;x_1,\dots,x_k)
	\end{equation*}
	for $x_1,\dots,x_k \ge 0$ is  additive in each of its $k$ variables.
	We  extend $A$ one
	variable at a time to a mapping
	$\tilde A \colon E^k \to F$.
	Thus, $\tilde A(x_1,x_2,\dots,x_m)$ is defined to be
	$A(x_1^+,x_2,\dots,x_m)-A(x_1^-,x_2,\dots,x_m)$
	for $x_2,\dots,x_m\ge 0$, and so on.  It is easy to see
	that this extension is additive in each variable and unique.
	
	Let $P$ be the $k$-homogeneous $\Q$-polynomial
	generated by $\tilde A$.  We claim that $P$ extends $f$.  
	To see this, let $x\in \plus{E}$.  Then, using the
	$k$-homogeneity condition (ii),
	\begin{align*}
	P(x) = A(x^k) &= \frac{1}{k!}\diff^k f(0,x^k)
	= \frac{1}{k!}\sum_{j=0}^k (-1)^{k-j}\binom{k}{j}f(jx)\\
	&= \frac{1}{k!} \sum_{j=0}^k (-1)^{k-j} \binom{k}{j} j^k \;f(x)
	=f(x)\,,
	\end{align*}
	using the identity
	\begin{equation*}
	\stirling{m}{k} = \frac{1}{k!}\sum_{j=0}^k (-1)^{k-j} 
	\binom{k}{j} j^m
	\end{equation*}
	with $k=m$. 
\end{proof}

\begin{proposition}[Kantorovich extension theorem for polynomials]
	Let $E$, $F$ be Riesz spaces with $F$ archimedean and let $f\colon \plus{E} \to \plus{F}$ 
	be a mapping that satisfies
	\begin{enumerate}
		\item[(i)] $\diff ^{m+1} f(x;h^{m+1})=0$ for all $x, h\in \plus{E}$, 
		and 
		\item[(ii)] $\diff ^kf(x;h_1,\dots,h_k)\ge 0$ for all 
		$x,h_1,\dots,h_k \in \plus{E}$, $1\le k\le m$\, ,
	\end{enumerate}
	for some $m\in \N$.
	Then $f$ extends to a unique positive polynomial $P\colon E\to F$ 
	of degree $m$ or less.
\end{proposition}

\begin{proof}
	Arguing as in the proof of the lemma, we take the Newton expansion
	of $f$ at $0$, using the vanishing of  the differences of order greater
	than $m$ to get 
	\begin{equation*}
	f(nx) = \sum_{j=0}^m \binom{n}{j} \diff^j f(0;x^j)
	\end{equation*}
	for $n\in \N$ and $x\ge 0$.  Using (\ref{stirling1}) to 
	represent the binomial coefficients in terms of powers of
	n, we see that this expression
	can be rearranged into the form
	\begin{equation*}
	f(nx) = \sum_{k=0}^m f_k(x) n^k
	\end{equation*}
	where
	\begin{equation*}
	f_k(x) = \sum_{j=k}^m \frac{1}{j!} \stirlingfirst{j}{k} (-1)^{j-k} 
	\diff^j f(0;x^j)\,.
	\end{equation*}
	From the fact that $f(n(px)) = f((np)x)$ for $n,p\in\N$, we get that
	$f_k(px) = p^k f_k(x)$ for each $k$. 
	And it follows from the interpolation formula (\ref{interpolation})
	that the mappings $f_k$ satisfy the condition
	$\diff^{m+1}f_k(0,h^{m+1}) =0$ for every $x,h\ge 0$.
	Together, these two facts imply that 
	$\diff^{k+1}f_k(0,h^{k+1}) =0$ for every $x,h \ge 0$.
	
	It follows from the lemma that 
	each $f_k$ extends to a $k$-homogeneous $\Q$-polynomial
	$P_k$. Thus $P= \sum_k P_k$
	is a $\Q$-polynomial that agrees with $f$ on $\plus{E}$.
	By condition (ii), $P$ is positive and 
	so, by Proposition \ref{PosProp}, $P$ is a polynomial.
\end{proof}



\bibliographystyle{plain}
\bibliography{PolynomialKantorovich}

\end{document}